\documentclass[12pt]{article}
\usepackage{amsmath,amssymb,amsbsy,amsfonts,amsthm,latexsym,
            amsopn,amstext,amsxtra,euscript,amscd,amsthm}

\newtheorem{lem}{Lemma}
\newtheorem{lemma}[lem]{Lemma}

\newtheorem{prop}{Proposition}

\newtheorem{thm}{Theorem}
\newtheorem{theorem}[thm]{Theorem}

\def\mand{\qquad\mbox{and}\qquad}

\def\\{\cr}
\def\({\left(}
\def\){\right)}
\def\[{\left[}
\def\]{\right]}
\def\<{\langle}
\def\>{\rangle}
\def\lcm{\mathrm{lcm}\,}
\def\fl#1{\left\lfloor#1\right\rfloor}

\def\cA{{\mathcal A}}

\def\cL{{\mathcal L}}
\def\cM{{\mathcal M}}
\def\cN{{\mathcal N}}
\def\cR{{\mathcal R}}
\def\cS{{\mathcal S}}
\def\cP{{\mathcal P}}
\def\cQ{{\mathcal Q}}

\def\N{{\mathbb N}}
\def\Z{{\mathbb Z}}
\def\K{{K}}
\def\Q{{\mathbb Q}}

\def\C{{\mathbb C}}

\begin{document}

\title{On numbers $n$ dividing the $n$th term of a linear recurrence}

\author{
{\sc Juan~Jos\'e~Alba~Gonz\'alez}\\ 
Instituto de Matem{\'a}ticas,\\
Universidad Nacional Autonoma de M{\'e}xico,\\
C.P. 58089, Morelia, Michoac{\'a}n, M{\'e}xico\\
{\tt jjalba@gmail.com}
\and
{\sc Florian~Luca}\\ 
Instituto de Matem{\'a}ticas,\\
Universidad Nacional Autonoma de M{\'e}xico,\\
C.P. 58089, Morelia, Michoac{\'a}n, M{\'e}xico\\
{\tt fluca@matmor.unam.mx}
\and
{\sc Carl~Pomerance}\\
Mathematics Department, 
Dartmouth College,\\
Hanover, NH 03755, USA\\
{\tt carl.pomerance@dartmouth.edu}
\and
{\sc Igor E. Shparlinski}  \\
Dept. of Computing, Macquarie University \\
Sydney, NSW 2109, Australia\\
{\tt igor.shparlinski@mq.edu.au}
}

\date{}

\pagenumbering{arabic}

\maketitle

\begin{abstract}
Here, we give upper and lower bounds on the count of positive integers $n\le x$ dividing the $n$th term of a nondegenerate linearly recurrent sequence with simple roots. 
\end{abstract}

\section{Introduction}

 Let $\{u_n\}_{n\ge 0}$ be a linear recurrence sequence of integers
satisfying a homogeneous linear recurrence relation
\begin{equation}
\label{eq:LRS}
u_{n+k} = a_1 u_{n+k-1} + \cdots + a_{k-1} u_{n+1} + a_k u_n,
\qquad {\text{\rm for}}\quad n =0,1, \ldots,
\end{equation}
where $a_1,\ldots, a_k$ are integers with $a_k \ne 0$. 

In this paper, we study the  set of indices $n$ which divide the
corresponding term $u_n$; that is, the set:
$$
\cN_u : = \{n\ge 1 \ : \  n\mid u_n\}.
$$
But first, some background on linear recurrence sequences.

To the recurrence~\eqref{eq:LRS}, we associate its {\it characteristic polynomial}
$$
f_u(X) : = X^k - a_1X^{k-1} - \cdots - a_{k-1}X - a_k =\prod_{i=1}^m (X-\alpha_i)^{\sigma_i}\in \Z[X],
$$
where $\alpha_1,\ldots,\alpha_m\in\C$ are the distinct roots of $f_u(X)$ with multiplicities 
$\sigma_1,\ldots,\sigma_m$, respectively.
It is then well-known that the general term of the recurrence can be expressed as
\begin{equation}
\label{eq:un}
u_n=\sum_{i=1}^m A_i(n)\alpha_i^n,\qquad {\text{\rm for}}\quad n=0,1,\ldots,
\end{equation}
where $A_i(X)$ are polynomials of degrees at most $\sigma_i-1$
for $i=1,\ldots,m$, with coefficients in
$K:=\Q[\alpha_1,\ldots,\alpha_m]$. We refer to~\cite{EvdPSW} for this and other known facts about
linear recurrence sequences.

For upper bounds on the distribution of $\cN_u$,
the case of a linear recurrence with multiple roots already can pose problems (but see below). 
For example, the sequence of general term 
$u_n=n2^n$ for all $n\ge 0$ having characteristic polynomial $f_u(X)=(X-2)^2$ shows that $\cN_u$ may contain all the positive integers.
So, we look at the case when $f_u(X)$ has only simple roots. In this case, the relation~\eqref{eq:un} becomes
\begin{equation}
\label{eq:un1}
u_n=\sum_{i=1}^k A_i\alpha_i^n,\qquad {\text{\rm for}}\quad n=0,1,\ldots.
\end{equation}
Here, $A_1,\ldots,A_k$ are constants in $K$. We may assume that none of them is zero, since otherwise, a little bit of Galois theory shows that 
the integer sequence $\{u_n\}_{n\ge 0}$ satisfies a linear recurrence of a smaller order. 

We remark in passing that there is no real obstruction in reducing to the case of the simple roots. 
Indeed, let $D\in\N$ be a common denominator of all the coefficients of all the polynomials $A_i(X)$ 
for $i=1,\ldots,m$. That is, the coefficients of each $DA_i$ are algebraic integers.  Then 
$$
Du_n=\sum_{i=1}^m DA_i(0) \alpha_i^n+\sum_{i=1}^m D(A_i(n)-A_i(0))\alpha_i^n.
$$
If $n\in \cN_u$, then $n\mid Du_n$. Since certainly $n$ divides\footnote{Here, for two algebraic integers $\alpha$ and $\beta$ and a positive integer $m$ we write $\alpha\equiv \beta\pmod m$ to mean that $(\alpha-\beta)/m$ is an algebraic integer. When $\beta=0$ we say that $m$ {\it divides} $\alpha$.}{} the algebraic integer
$$
\sum_{i=1}^m D(A_i(n)-A_i(0))\alpha_i^n,
$$
it follows that $n$ divides $\sum_{i=1}^m DA_i(0)\alpha_i^n$. If this is identically zero (i.e., $A_i(0)=0$ for all $i=1,\ldots,m$), then we are in an instance similar to the instance of the sequence of general term $u_n=n2^n$ for all $n\ge 0$ mentioned above. In this case, $\cN_u$ contains at least a positive proportion of all the positive integers (namely, all $n$ coprime to $D$). Otherwise, we may put
$$
w_n=\sum_{i=0}^m DA_i(0)\alpha_i^n\qquad {\text{\rm for}}\quad n=0,1,\ldots.
$$
A bit of Galois theory shows that $w_n$ is an integer for all $n\ge 0$, and the sequence 
$\{w_n\}_{n\ge 0}$ satisfies a linear recurrence relation of order $\ell:=\#\{1\le i\le m: A_i(0)\ne 0\}$ with integer coefficients, 
which furthermore has only simple roots. Hence, $\cN_u\subseteq \cN_w$, and therefore there is indeed no 
loss of generality when proving upper bounds
in dealing only with linear recurrent sequences with distinct roots.

We put 
\begin{equation}
\label{eq:Delta}
\Delta_u:=\prod_{1\le i<j\le k} (\alpha_i-\alpha_j)^2={\text{\rm disc}}(f_u)
\end{equation}
for the (nonzero) discriminant of the sequence 
$\{u_n\}_{n\ge 0}$, or of the polynomial $f_u(X)$. It is known that $\Delta_u$ is an integer. 
We  also assume that $(u_n)_{n\ge 0}$ is nondegenerate, 
which means that $\alpha_i/\alpha_j$ is not a root of $1$
for any $1\le i<j\le m$. Throughout the remainder of this paper, all linear recurrences have only 
simple roots and are nondegenerate. 

When $k=2,~u_0=0$, $u_1=1$ and $\gcd(a_1,a_2)=1$, the sequence $\{u_n\}_{n\ge 0}$ is  called a {\it Lucas sequence}. The formula~\eqref{eq:un1} of its general term is 
\begin{equation}
\label{eq:unLucas}
u_n=\frac{\alpha_1^n-\alpha_2^n}{\alpha_1-\alpha_2},\qquad {\text{\rm for}}\quad n=0,1,\ldots.
\end{equation}
That is, we can take $A_1=1/(\alpha_1-\alpha_2)$ and $A_2=-1/(\alpha_1-\alpha_2)$ in the formula of the general term~\eqref{eq:un1}.
In the case of a Lucas sequence $(u_n)_{n\ge 0}$, the fine structure 
of the set $\cN_u$ has been described in~\cite{S2} (see the references  
therein and particularly~\cite{GS}).  We also note that divisibility of 
terms of a linear recurrence sequence by arithmetic functions of their 
index have been studied in~\cite{LuShp} (see also~\cite{Luc} for the special
case of Fibonacci numbers). 

For a set $\cA$ and a positive real number $x$ we put $\cA(x)=\cA\cap [1,x]$.  Throughout the paper, we study upper and lower bounds for the number $\#\cN_u(x)$. In particular, we prove that $\cN_u$ is of asymptotic density zero.  

Observe first that if $k=1$, then $u_n=A a_1^n$ holds for all $n\ge 0$ with some integers $A\ne 0$ and $a_1\not\in \{0,\pm 1\}$. Its characteristic polynomial is $f_u(X)=X-a_1$. It is easy to see that in this case $\#\cN_u(x)=O((\log x)^{\omega(|a_1|)})$, where for an integer $m\ge 2$ we use $\omega(m)$ for the number of distinct prime factors of $m$. So, from now on, we assume that $k\ge 2$. 

Note next that for the 
sequence of general term $u_n = 2^{n} -2$ for all $n\ge 0$ having
characteristic polynomial $f_u(X) = (X-1)(X-2)$, Fermat's little theorem implies that
every prime is in $\cN_u$, so that the Prime Number Theorem and estimates for the
distribution of pseudoprimes\footnote{A pseudoprime is a composite number $n$ which divides $2^n-2$.
The paper~\cite{POM} shows that there are few odd pseudoprimes compared with primes, while~\cite{Li}
(unpublished) does the same for even pseudoprimes.}
show that it is possible for the estimate $\#\cN_u(x)=(1+o(1))x/\log x$ to hold as $x\to\infty$. 
However, we  show that $\#\cN_u(x)$ cannot have a larger order of magnitude.

\begin{theorem}
\label{thm:Div}
For each $k\ge 2$, there is a positive constant $c_0(k)$ depending only on $k$ such that if the characteristic polynomial of a nondegenerate  linear recurrence
sequence  $\{u_n\}_{n \ge 0}$ of order $k$ has only simple roots, then the estimate 
$$
\#\cN_u(x)  \le c_0(k)\frac{x}{\log x}
$$
holds for $x$ sufficiently large.
\end{theorem}

In case of a Lucas sequence, we have a better bound. To simplify notations, for a posititive integer $\ell$ we define $\log_\ell x$ iteratively as $\log_1 x:=\max\{\log x, 1\}$ and for $\ell>1$ as $\log_\ell x:=\log_1 (\log_{\ell-1} x)$. When $\ell=1$ we omit the index but understand that all logarithms are $\ge 1$.
 Let
\begin{equation}
\label{eq:Lofx}
L(x): =\exp\left(\sqrt{\log x \log_2 x}\right).
\end{equation}

\begin{theorem}
\label{thm:2}
Assume that $\{u_n\}_{n\ge 0}$ is a Lucas sequence. Then the inequality
\begin{equation}
\label{eq:upperJ2}
\#\cN_u(x)\le \frac{x}{L(x)^{1+o(1)}}
\end{equation}
holds as $x\to\infty$.
\end{theorem}

It follows from a result of Somer~\cite[Theorem~8]{Som} 
that $\cN_u$ is finite if and only if $\Delta_u=1$, and
in this case $\cN_u=\{1\}$.

For Lucas sequences  with $a_2 =\pm 1$,  we also have a 
rather strong lower bound on $\#\cN_u(x)$. Our result depends on 
the current knowledge of the distribution of {\it $y$-smooth} values of 
$p^2-1$ for primes $p$, that is values of $p^2-1$ which do not  have 
prime divisors exceeding $y$. We use  $\Pi(x,y)$ to denote the number of primes $p\le x$
for which $p^2-1$ is  $y$-smooth.
 Since the numbers $p^2-1$ with $p$ prime are likely to behave as ``random'' integers from the point of view of the size of their prime factors,
it seems reasonable to expect that behavior of  
$\Pi(x,y)$ resembles the behavior of the counting function 
for smooth integers. We record this  in a very relaxed 
form of the assumption that for some fixed 
real $v \ge  1$ we have 
\begin{equation}
\label{eq:asymp}
\Pi(y^v,y) \ge  y^{v+o(1)}
\end{equation}
as $y \to \infty$.
In fact, a general result from~\cite[Theorem~1.2]{DMT}
implies that~\eqref{eq:asymp} holds with any $v \in [1,4/3)$.

\begin{theorem}
\label{thm:lowerboundJ}
There is a set of integers $\cL$ such that $\cL\subset\cN_u$
for any Lucas sequence $u$ with $a_2=\pm1$, 
and such that if~\eqref{eq:asymp} holds with some $v>1$, we have 
$$
\#\cN_u(x) \ge \#\cL(x)\ge x^{\vartheta+o(1)}
$$
as $x\to\infty$, where
$$
\vartheta: = 1 - 1/v.
$$
\end{theorem}

In particular, since as we have already mentioned, any value of $v < 4/3$
is admissible, we can take
$$
\vartheta = 1/4.
$$
Furthermore, since~\eqref{eq:asymp}  is expected to hold for
{\it any} $v>1$, it is very likely that the bound of 
Theorem~\ref{thm:lowerboundJ}
holds with $\vartheta = 1$. 

Finally, we record a lower bound on $\#\cN(x)$ when $a_2\ne \pm 1$ but $\Delta_u\ne 1$.

\begin{theorem}
\label{thm:lucasgen}
Let $\{u_n\}_{n\ge 0}$ be any Lucas sequence with $\Delta_u\ne 1$. Then there exist positive constants $c_1$ and $x_0$ depending on the sequence such that for $x>x_0$ we have  
$$
\#\cN_u(x)>\exp(c_1(\log_2 x)^2).
$$
\end{theorem}

Throughout the paper, we use $x$ for a large positive real number. We use the Landau symbol $O$ and the Vinogradov symbol $\ll$ with the usual meaning in analytic number theory. 
The constants implied by them may depend on the sequence $\{u_n\}_{n\ge 0}$, or only on $k$. 
We use $c_0,c_1,\ldots$ for positive constants  which may depend on $\{u_n\}_{n\ge 0}$. 
We label such constants increasingly as they appear in the paper.

\section{Preliminary results}
\label{sec:prelim}

As in the proof of~\cite[Theorem~2.6]{EvdPSW}, put
$$
D_u(x_1,\ldots,x_k):=\det(\alpha_i^{x_j})_{1\le i, j\le k}.
$$
For a prime number $p$ not dividing $a_k$, let $T_u(p)$ be the maximal nonnegative integer $T$ with the property that
$$
p\nmid \prod_{0\le x_2,\ldots,x_k\le T} \max\{1,|N_{\K/\Q}(D_u(0,x_2,\ldots,x_k))|\}.
$$
It is known that such $T$ exists. In the above relation, $x_2,\ldots,x_k$ are integers in $[1,T]$, 
and for an element $\alpha$ of $\K$ we use $N_{\K/\Q}(\alpha)$ for the norm of $\alpha$ over $\Q$. 
Since $\alpha_1,\ldots,\alpha_k$ are algebraic integers in $\K$, it follows that the numbers 
$N_{\K/\Q}(D_u(0,x_2,\ldots,x_k))$ are integers.

Observe that $T_u(p)=0$ if and only if $k=2$ and $p$ is a divisor of $\Delta_u=(\alpha_1-\alpha_2)^2$.

More can be said in the case when $\{u_n\}_{n\ge 0}$ is a Lucas sequence. 
In this case, we have 
$$
|N_{\K/\Q}(D_u(0,x_2))|=|\alpha_2^{x_2}-\alpha_1^{x_2}|^2=|\Delta_u|^2 |u_{x_2}|^2,\qquad x_2=1,2,\ldots.
$$
Thus, if $p$ does not divide the discriminant $\Delta_u=(\alpha_1-\alpha_2)^2=a_1^2+4a_2$ of the sequence $\{u_n\}_{n\ge 0}$, 
then $T_u(p)+1$ is in fact the minimal positive integer $\ell$ such that $p\mid u_\ell$. This is sometimes called 
the {\it index of appearance\/} of $p$ in $\{u_n\}_{n\ge 0}$  and is denoted by $z_u(p)$. The index of appearance $z_u(m)$ can be defined for composite integers $m$ in the same way as above, namely as the minimal positive integer $\ell$ such that $m\mid u_\ell$. 
This exists for all positive integers $m$ coprime to $a_2$, and has the important property that $m\mid u_{n}$ if and only if $z_u(m)\mid n$.
For any $\gamma\in (0,1)$, let 
$$
\cP_{u,\gamma}=\{p~:~T_u(p)<p^{\gamma}\}.
$$

\begin{lem}
\label{lem:3}
For $x^\gamma,y\ge2$, the estimates
$$
\#\{p:T_u(p)\le y\}\ll\frac{y^k}{\log y},\quad
\#\cP_{u,\gamma} (x)\ll \frac{x^{k\gamma}}{\gamma\log x}
$$
hold, where the implied constants depend only on the sequence 
$\{u_n\}_{n\ge 0}$.
\end{lem}

\begin{proof}
It is clear that the second inequality follows immediately from the
first with $y=x^\gamma$, so we prove only the first one.  Suppose that
$T_u(p)\le y$.
In particular, there exists a choice of integers $x_2,\ldots,x_k$ all 
in $[1,y+1]$ such that
$$
p\mid \max\{1,|N_{\K/\Q}(D_u(0,x_2,\ldots,x_k))|\}.
$$
This argument shows that
\begin{equation}
\label{eq:TTT}
\prod_{T_u(p)\le y} p\mid \prod_{1\le x_2,\ldots,x_k\le y+1}   
\max\{1,|N_{\K/\Q}(D_u(0,x_2,\ldots,x_k))|\}.
\end{equation}
There are at most $(y+1)^{k-1}=O(y^{k-1})$ possibilities for the 
$(k-1)$-tuple $(x_2,\ldots,x_k)$. 
For each one of these $(k-1)$-tuples, we have that
$$
|N_{\K/\Q}(D_u(0,x_2,\ldots,x_k))|=\exp(O(y)).
$$
Hence, the right hand side in~\eqref{eq:TTT} is of size $\exp(O(y^{k}))$.
Taking logarithms in the inequality implied by~\eqref{eq:TTT}, we get that
\begin{equation*}
\sum_{T_u(p)\le y}\log p =O(y^{k}).
\end{equation*}
If there are a total of $n$ primes involved in this sum and if $p_i$ denotes
the $i$th prime, then 
$$
\sum_{i=1}^n\log p_i=O(y^k),
$$
so that in the language of the prime number theorem, $\theta(p_n)\ll y^k$.
It follows that $p_n\ll y^k$ and $n\ll y^k/(k\log y)$, which is what
we wanted to prove.
\end{proof}

The parameter $T_u(p)$ is useful to bound the number of solutions $n\in [1,x]$ of the congruence $u_n\equiv 0\pmod p$. For example, the following 
is~\cite[Theorem~5.11]{EvdPSW}.

\begin{lemma}
\label{lem:T}
There exists a constant $c_2(k)$ depending only on $k$ with the following property. 
Suppose that $\{u_n\}_{n\ge 0}$ is a linearly recurrent sequence of order $k$ satisfying recurrence 
\eqref{eq:LRS}. Suppose that $p$ is a prime coprime to $a_k\Delta_u$. 
Assume that there exists a positive integer $s$ such that $u_{s}$ is coprime to $p$. 
Then for any real $x\ge 1$ the number of solutions $R(x,p)$ of  the congruence
$$
u_n\equiv 0\pmod p\qquad {\text{with}}\quad 1\le n\le x
$$
satisfies the bound
$$
R_u(x,p)\le c_2(k)\left(\frac{x}{T_u(p)}+1\right).
$$
\end{lemma}

When $\{u_n\}_{n\ge 0}$ is a Lucas sequence, we put 
$$
\cQ_{u,\gamma}=\{p~:~z_u(p)\le p^{\gamma}\}.
$$

The remarks preceding Lemma~\ref{lem:3} show that $\#\cQ_{u,\gamma}(x)=\#\cP _{u,\gamma}(x)+O(1)$. Hence, Lemma~\ref{lem:3} implies the following result.

\begin{lem}
\label{lem:4}
For $x>1$, the estimate 
$$
\#\cQ_{u,\gamma} (x)\ll \frac{x^{2\gamma}}{\log x}
$$
holds, where the implied constant depends only on the sequence $\{u_n\}_{n\ge 0}$.
\end{lem}

As usual, we denote by $\Psi(x,y)$  the number of integers $n\le x$ with $P(n)\le y$.
By~\cite[Corollary to Theorem~3.1]{CEP}, we have the following 
well-known result. 

\begin{lem}
\label{lem:Smooth}
For $x\ge y >1$, the estimate 
$$
\Psi(x,y) = x\exp(-(1+o(1))v\log v)
$$
uniformly in the range $y>(\log x)^2$ as long as $v\to\infty$, where 
$$
v:=(\log x)/(\log y).
$$
\end{lem}

\section{The proof of Theorem~\ref{thm:Div}}

We assume that $x$ is large. We split the set $\cN_u(x)$ into several subsets. 
Let $P(n)$ be the largest prime factor of $n$ and let $y:=x^{1/\log\log x}$.
Let 
\begin{eqnarray*}
\cN_1(x) & := & \{n\le x~:~P(n)\le y\};\\
\cN_2(x) & := & \{n\le x~:~n\not\in \cN_1(x)~{\text{\rm and}}~P(n)
\in \cP_{u,1/(k+1)}\};\\
\cN_3(x) & := & \cN(x)\backslash \left(\cup_{i=1}^2 \cN_i(x)\right).
\end{eqnarray*}

We now bound the cardinalities of each one of the above sets.

For $\cN_1(x)$, by Lemma~\ref{lem:Smooth}, we obtain
\begin{equation}
\label{eq:1}
\#\cN_1(x)= \Psi(x,y) = x\exp(-(1+o(1))v\log v) = o\(\frac{x}{\log x}\)
\end{equation}
as $x\to\infty$, where 
$$v=\frac{\log x}{\log y} = \log \log x .
$$ 

Suppose now that $n\in \cN_2(x)$. Then $n=pm$, where $p=P(n)\ge \max\{y,P(m)\}$. In particular, 
$p\le x/m$ therefore $m\le x/y$. 
Since we also have $p\in \cP_{u,1/(k+1)}(x/m)$, 
Lemma~\ref{lem:3} implies that the number of such primes $p\le x/m$ is 
$O\(\left(x/m\right)^{k/(k+1)}\)$, where 
the implied constant  depends on the sequence $\{u_n\}_{n\ge 0}$. 
Summing up the above inequality over all possible values of $m\le x/y$, we get 
\begin{equation}
\begin{split}
\label{eq:3}
\#\cN_2(x) & \le  x^{k/(k+1)}\sum_{1\le m\le x/y} \frac{1}{m^{k/(k+1)}}\ll x^{k/(k+1)} \int_{1}^{x/y} \frac{dt}{t^{k/(k+1)}}\\
& =  ((k+1)x^{k/(k+1)}) t^{1/(k+1)} \Big|_1^{x/y}\ll \frac{x}{y^{1/(k+1)}}.
\end{split}
\end{equation}

Now let $n\in \cN_3(x)$. As previously, we write $n=pm$, where $p=P(n)>y$. We assume that $x$ (hence, $y$) is sufficiently large. Thus, $m\le x/p< x/y$. 
Since $n\in \cN_u$, we have that $n\mid u_n$, therefore $p\mid u_n$.  
Furthermore, $T_u(p)\ge p^{1/(k+1)}$. 
We fix $p$ and count the number of possibilities for $m$. To this end, 
let $\{w_\ell\}_{\ell\ge 0}$ be the sequence defined as $w_\ell=u_{p\ell}$ for all $\ell\ge 0$.   This is a linearly recurrent sequence of order $k$. We would like to apply Lemma~\ref{lem:T} to it to bound the number of solutions to the congruence
$$
w_m\equiv 0\pmod p,\qquad {\text{\rm where}}\quad 1\le m\le x/p.
$$
If the conditions of Lemma~\ref{lem:T} are satisfied, then this number denoted by $R_w(x/p,p)$ satisfies
$$
R_w(x/p,p)\le c_2(k)\left(\frac{x}{pT_w(p)}+1\right).
$$
Let us check the conditions of Lemma~\ref{lem:T}.
Note first that if $\alpha_1,\ldots,\alpha_k$ are the characteristic roots of 
$\{u_n\}_{n\ge 0}$, then $\alpha_1^p,\ldots,\alpha_k^p$ are the 
characteristic roots of 
$\{w_\ell\}_{\ell\ge 1}$. Hence, 
$$
f_w(X)=\prod_{i=1}^k (X-\alpha_i^p).
$$
In particular, the term $a_{w,k}$ corresponding to the recurrence $\{w_{\ell}\}_{\ell\ge 1}$ satisfies $a_{w,k}=a_{k}^p$ assuming that $y>2$. 
Thus, assuming further that $y>|a_{k}|$, we then have that $p$ does not divide $a_{k}$, therefore $p$ does not divide $a_{w,k}$ either. 
Next, note that
$$
\Delta_w=\prod_{1\le i<j\le k} (\alpha_i^p-\alpha_j^p).
$$
Modulo $p$, we have that 
$$
\Delta_w\equiv \left(\prod_{1\le i<j\le k} (\alpha_i-\alpha_j)\right)^p\equiv \Delta_u^p\pmod p.
$$
From the above congruence, we easily get that $p\mid \Delta_w$ if and only if $p\mid \Delta_u$. Thus, assuming that $x$ is sufficiently large such that $y>|\Delta_u|$, we then  have that $p\nmid \Delta_u$, therefore $p\nmid \Delta_w$ either.  

So far, we have checked that $p$ does not divide $a_{w,k}\Delta_w$, which is the first assumption in the statement of Lemma~\ref{lem:T}. 

Let us check the next assumption. 

Note that since $p\nmid \Delta_u$, the characteristic polynomial $f_u(X)$ of $\{u_\ell\}_{\ell\ge 0}$ has only simple roots modulo $p$.  Since $p$ does not divide the last coefficient $a_{k}$ for the recurrence for $\{u_n\}_{n\ge 0}$ either, it follows that this sequence is purely periodic modulo $p$. Let $t_p$ be its period modulo $p$. It is known that $t_p$ is coprime to $p$. In fact, $t_p$ is a divisor 
of the number 
$$
\lcm[p^i-1:i=1,2,\ldots,k].
$$ 
Choose some $n_0>0$ such that $u_{n_0}\ne 0$. Let $x$ be so large such that $y>|u_{n_0}|$. Since $p>y$, we have $p\nmid u_{n_0}$. And since 
$\gcd(p,t_p)=1$, there exists an integer $s$ with $sp\equiv n_0\pmod {t_p}$. Thus,
$$
w_s = u_{sp}\equiv u_{n_0} \pmod p.
$$
In particular, $w_s$ is coprime to $p$. Hence, for $x$ sufficiently large, 
the second assumption from Lemma~\ref{lem:T} holds 
for the sequence $\{w_\ell\}_{\ell \ge 0}$.

Next we show that $T_u(p)=T_w(p)$. Observe that this number exists (both for the sequence $\{u_\ell\}_{\ell \ge 0}$ and $\{w_{\ell}\}_{\ell\ge 0}$) because $p$ does not divide $a_k$. Indeed, the claimed equality follows easily from the following calculation:
\begin{eqnarray*}
D_w(x_1,\ldots,x_k) & = & {\text{\rm det}}(\alpha_i^{px_j})_{1\le i,j\le k}\equiv ({\text{\rm det}}(\alpha_i^{x_j}))^p \pmod p\\
& \equiv &  D_u(x_1,\ldots,x_k)^p\pmod p.
\end{eqnarray*}
Since $n\in \cN_3(x)$, we have that $T_w(p)=T_u(p)\ge p^{1/(k+1)}$. 

Lemma~\ref{lem:T} now guarantees that the number of choices for $m$ once $p$ is fixed is
$$
R_w(x/p,p)\le c_2(k)\left(\frac{x}{p^{1+1/(k+1)}}+1\right).
$$
To summarize, we have
\begin{eqnarray*}
\cN_3(x) & \le & \sum_{y\le p\le x} c_2(k)\left(\frac{x}{p^{1+1/(k+1)}}+1\right)\\
& \le & c_2(k)\left(\pi(x)+x\sum_{y\le p} \frac{1}{p^{1+1/(k+1)}}\right)\\
& \le & c_2(k)\left(\pi(x)+x\int_{y}^{\infty} \frac{dt}{t^{1+1/(k+1)}}\right)\\
\end{eqnarray*}
Therefore
\begin{equation}
\label{eq:4}
\cN_3(x)  \le 
c_2(k)\left(\pi(x)+O\(\frac{(k+1)x}{y^{1/(k+1)}}\)\right).
\end{equation}
Comparing~\eqref{eq:1}, \eqref{eq:3} and~\eqref{eq:4}, we get that
\begin{equation}
\label{eq:final}
\#\cN(x)\le c_2(k)\pi(x)+\frac{x}{\exp((1+o(1))v\log v)}+O\left(\frac{x}{ y^{1/(k+1)}}\right)
\end{equation}
as $x\to\infty$,
where the implied constant depends on the recurrence $\{u_n\}_{n\ge 0}$.
By our choice of $y$ as $x^{1/\log\log x}$, the second and third terms
on the right side of \eqref{eq:final} are both $o(\pi(x))$ as $x\to\infty$,
so we have the theorem.

\section{The proof of Theorem~\ref{thm:2}}

We divide the numbers $n\in \cN_u(x)$ into several classes:
\begin{itemize}
\item[(i)] $\cN_1(x):=\{n\in\cN_u(x): P(n)\le L(x)^{1/2}\}$;
\item[(ii)] $\cN_2(x):=\{n\in \cN_u(x):P(n)\ge L(x)^3\}$;
\item[(iii)] $\cN_3(x):=\cN_u(x)\setminus(\cN_1(x)\cup\cN_2(x))$.
\end{itemize}

It follows from Lemma~\ref{lem:Smooth}  that
$$
\#\cN_1(x)\le\Psi(x,L(x)^{1/2})\le\frac{x}{L(x)^{1+o(1)}}
$$
as $x\to\infty$.

For $n\in\cN_u$ and $p\mid n$, we have $n\equiv0\pmod p$ and 
$n\equiv0\pmod{z_u(p)}$.  For $p$ not dividing the discriminant of the
characteristic polynomial for $u$
(and so for $p$ sufficiently large), we have $z_u(p)\mid p\pm1$, so that
$\gcd(p,z_u(p))=1$.  Thus, the conditions $n\in\cN_u$, $p\mid n$,
and $p$ sufficiently large jointly force $n\equiv0\pmod{pz_u(p)}$.
Hence, if $p$ is sufficiently large, the number of $n\in\cN_u(x)$ with
$P(n)=p$ is at most $\Psi(x/pz_u(p),p)\le x/pz_u(p)$.   

Thus, for large $x$,
$$
\#\cN_2(x)\le\sum_{p>L(x)^3}\frac{x}{pz_u(p)}=
\sum_{\substack{p>L(x)^3\\ z_u(p)\le L(x)}}\frac{x}{pz_u(p)}
+\sum_{\substack{p>L(x)^3\\ z_u(p)> L(x)}}\frac{x}{pz_u(p)}.
$$
The first sum on the right has, by Lemma~\ref{lem:3}, at most $L(x)^2$ terms
for $x$ large, each term being smaller than $x/L(x)^3$, so the
sum is bounded by $x/L(x)$.  The second sum on the right has terms
smaller than $x/pL(x)$ and the sum of $1/p$ is of magnitude $\log\log x$,
so the contribution here is $x/L(x)^{1+o(1)}$ as $x\to\infty$.
Thus, $\#\cN_2(x)\le x/L(x)^{1+o(1)}$ as $x\to\infty$.

For any nonnegative integer $j$, let $I_j:=[2^j,2^{j+1})$.
For $\cN_3$, we cover $I:=[L(x)^{1/2},L(x)^3)$ by these dyadic intervals,
and we define $a_j$ via $2^j=L(x)^{a_j}$.  We shall assume the variable
$j$ runs over just those integers with $I_j$ not disjoint from $I$.
For any integer $k$, define $\cP_{j,k}$ as the set of primes $p\in I_j$
with $z_u(p)\in I_k$.  Note that, by Lemma~\ref{lem:3}, we have $\#\cP_{j,k}\ll 4^k$.
We have
\begin{align*}
\#\cN_3(x)\le\sum_j\sum_k\sum_{p\in\cP_{j,k}}\sum_{\substack{n\in\cN_u(x)\\
P(n)=p}}1
&\le
\sum_j\sum_k\sum_{p\in\cP_{j,k}}\Psi\left(\frac{x}{pz_u(p)},p\right)\\
&\le \sum_j\sum_k\sum_{p\in\cP_{j,k}}\frac{x}{pz_u(p)L(x)^{1/2a_j+o(1)}},
\end{align*}
as $x\to\infty$, where we have used Lemma~\ref{lem:Smooth} 
for the last estimate.
For $k>j/2$, we use the estimate
$$
\sum_{p\in\cP_{j,k}}\frac1{pz_u(p)}\le2^{-k}\sum_{p\in I_j}\frac1p\le2^{-k}
$$
for $x$ large.  For $k\le j/2$, we use the estimate
$$
\sum_{p\in\cP_{j,k}}\frac1{pz_u(p)}\ll\frac{4^k}{2^j2^k}=2^{k-j},
$$
since there are at most order of magnitude $4^k$ such primes, as noted before.
Thus,
\begin{eqnarray*}
\sum_k\sum_{p\in\cP_{j,k}}\frac{1}{pz_u(p)}
& =& \sum_{k>j/2}\sum_{p\in\cP_{j,k}}\frac{1}{pz_u(p)}
+\sum_{k\le j/2}\sum_{p\in\cP_{j,k}}\frac{1}{pz_u(p)}\\
& \ll & 2^{-j/2}=L(x)^{-a_j/2}.
\end{eqnarray*}
We conclude that
$$
\#\cN_3(x)\le\sum_j\frac{x}{L(x)^{a_j/2+1/2a_j+o(1)}}\qquad {\text{\rm as}}\qquad x\to\infty.
$$  
Since the minimum value of $t/2+1/(2t)$ for $t>0$ is 1
occuring at $t=1$, we conclude that $\#\cN_3(x)\le x/L(x)^{1+o(1)}$
as $x\to\infty$.  With our prior estimates for $\#\cN_1(x)$ and $\#\cN_2(x)$,
this completes our proof.

It is possible that using the methods of \cite{ELP} and \cite{GP}
a stronger estimate can be made.

\section{The proof of Theorem~\ref{thm:lowerboundJ}}

Since $a_2=\pm 1$, it is easy to see that the sequence $u$ is purely 
periodic modulo any integer $m$. So, the index of appearance $z_u(m)$
defined in Section~\ref{sec:prelim} exists  for all positive integers $m$. 
Further, by examining 
the explicit formula~\eqref{eq:unLucas} one can see 
that for any prime power $q=p^k$ we have 
\begin{equation}
\label{eq:z div}
z_u(p^k) \mid z_u(p)p^{k-1}.
\end{equation}
In fact this is known in much wider generality. 

Now, for any real number $y\ge1$ let
$$
M_y:=\lcm[m:~m\le y].
$$
We say that a positive integer $n$ is {\it Lucas special\/}
if it is of the form $n = 2sM_y$ for some $y\ge3$ and for some
squarefree positive integer  $s$ such that $\gcd(s,M_y)=1$ and
for every prime $p\mid s$ we have $p^2-1\mid M_y$.  Let $\cL$
denote the set of Lucas special numbers.

We now show that $\cL\subset\cN_u$
for any Lucas sequence $u$ with $a_2=\pm1$.
To see this it suffices to show for any
$n=2sM_y\in\cL$ and for any prime power
$q\mid n$, we have $z_u(q) \mid n$.
This is easy for $q\mid s$, since then $q=p$ is prime and either
$z_u(p)=p$ (in the case $p\mid\Delta_u$) or
$z_u(p)\mid p\pm1$ (otherwise).  And since $p^2-1\mid M_y$, we have
$z_u(p)\mid n$ in either case.

If $q \mid 2M_y$, we consider the cases of odd and even $q$ separately. 
\begin{itemize}
\item When $q$ is odd, we have 
$q \mid M_y$ so $q \le y$. 
Write $q=p^k$ with $p$ prime, so that~\eqref{eq:z div} implies
$z_u(q)\mid (p-1)p^{k-1}, p^k$ or $(p+1)p^{k-1}$.  
We have $p^{k-1}\le y$ and if $p+1\le y$, then 
$z(q)\mid M_y$.  The only case not covered is $p+1>y$ (so $p\in(y-1,y]$), 
$k=1$, $z_u(p)=p+1$.
Write $p+1=2^jm$ where $m$ is odd.  
Then $2^j\mid2M_y$ and $m\mid2M_y$, so $p+1\mid2M_y$.
Thus, in all cases, $z_u(q)\mid 2M_y$ so $z_u(q)\mid n$.

\item  When $q = 2^k$ is a power of $2$ with $q\mid 2M_y$, then since $z_u(2)\in \{2,3\}$, 
we see from~\eqref{eq:z div} that either 
$z_u(2^k) \mid 2^k$ or $z_u(2^k) \mid 3\cdot 2^{k-1}$. Since $y\ge3$,
in either case we have $z_u(q) \mid  2 M_y$.

\end{itemize}

We now use the method of Erd\H os~\cite{Er} to show that the set 
$\cL$ is rather large.
For this we take 
$$
y := \frac{\log x}{\log \log x} \mand z := y^v.
$$
We say that $q$ is a {\it proper prime power} if 
$q = \ell^k$ for a prime $\ell$ and an integer $k\ge 2$.

We define $\cP$ as the set of primes  $p$ such that:
\begin{itemize}
\item  $p \in [y+1, z]$;
\item  $p^2-1$ is $y$-smooth;
\item  $p^2-1$ is not divisible by any 
proper prime power $q > y$. 
\end{itemize}

Note that if $q$ is a proper prime power and $q\mid p^2-1$,
then $q\mid p\pm1$, unless $q$ is even, in which case
$q/2\mid p\pm1$.
Since trivially there are only $O(t^{1/2})$ proper 
prime powers $q \le t$, there are only 
$O(z y^{-1/2})$ primes $p \le z$ for which $p^2-1$ is divisible by a
proper prime power $q > y$. 
Thus, recalling the assumption~\eqref{eq:asymp}, we obtain
$$
\#\cP \ge \Pi(z,y) - y + O(zy^{-1/2}) = z^{1 +o(1)},
$$
provided that $x\to \infty$. 

It is also obvious that for any squarefree positive integer $s$ 
composed out of primes $p \in \cP$, the integer $n = 2sM_y$ 
is Lucas special.

We now take the set $\cL_v(x)$ of all such Lucas special integers $n = 2sM_y$,
where  $s$ is composed out of 
$$
r :=\fl{\frac{\log x - 2 y}{\log z}}
$$ 
distinct primes $p\in \cP$. 
Since by the prime number theorem the estimate 
$M_y =\exp((1+o(1)) y)$ holds as $x\to\infty$, 
we see that for sufficiently large $x$ we have $n \le x$ for 
every $n \in \cL_v(x)$. 

For the cardinality of $\cL_v(x)$ we have
$$
\# \cL_v(x) \ge \binom{\#\cP}{r} \ge  \(\frac{\#\cP}{r}\)^r . 
$$
Since 
$$
r = (v^{-1} + o(1)) \frac{\log x}{\log \log x}
\mand 
\frac{\#\cP}{r} = (\log x) ^{v-1 + o(1)}
$$ 
as $x\to\infty$, we obtain $\# \cL_v(x) \ge x^{1-1/v + o(1)}$ as $x\to\infty$.
Noting that $\cL_v(x)\subset\cL(x)$
concludes the proof.  

\section{The proof of Theorem \ref{thm:lucasgen}}

Since $\Delta_u\equiv 0,1\pmod 4$ and $\Delta_u\ne 0,~1$, it follows that $|\Delta_u|>1$. Let $r$ be some prime factor of $\Delta_u$. 
Then $r^k\in \cN_u$ for all $k\ge 0$ (see \cite[pages~210 and~295]{Lucas}).
We let $k$ be a large positive integer and look at
$u_{r^{k+4}}$. By Bilu, Hanrot and Voutier's primitive divisor theorem (see \cite{BHV}), $u_n$ has a primitive prime factor for all $n\ge 31$. Recall that a primitive prime factor of $u_n$ is a prime factor $p$ 
of $u_n$ which does not divide 
$\Delta_u u_m$ for any positive integer $m<n$.   
Such a  primitive prime factor $p$   always satisfies
$p\equiv \pm 1\pmod n$. Since there are at most $5$ values of $k\ge 0$ such that $r^k\le 30$ for the same integer $r>1$, 
and since
$u_m \mid u_n$ if $m \mid n$, 
we conclude that $u_{r^{k+4}}$ has at least $\tau(r^{k+4})-5=k$ 
distinct prime factors $p\ne r$, where
 $\tau(m)$ is the number of divisors of the positive integer $m$. Let them be $p_1<\cdots<p_k$. Assume that $|\alpha_1|\ge |\alpha_2|$. For large $n$, we have that $|\alpha_1|^{n/2}<|u_n|<2|\alpha_1|^n$ 
(see~\cite[Theorem~2.3]{EvdPSW}).  If $\beta_1,\ldots,\beta_k$ are nonnegative
exponents such that
$$
\beta_i\le \frac{\log (x/r^{k+4})}{k\log p_i},
$$
then $r^{k+4} p_1^{\beta_1}\cdots p_k^{\beta_k}\le x$ is in 
$\cN_u$ (see~\cite[Page~210]{Lucas}), so it is counted by $\#\cN_u(x)$. Hence,
\begin{equation*}
\begin{split}
\#\cN_u(x)&\ge \prod_{i=1}^k\left(\left\lfloor
{{\log(x/r^{k+4})}\over {k\log p_i}}\right\rfloor+1\right)\ge
\left({{\log(x/r^{k+4})}\over {k}}\right)^k {{1}\over {\prod_{i=1}^k
\log p_i}}\\
&\ge \left({{\log(x/r^{k+4})}\over {2r^{k+4} \log |\alpha_1|}}\right)^k,
\end{split}
\end{equation*}
where the last inequality follows
from the mean value inequality
\begin{eqnarray*}
\prod_{i=1}^k \log p_i\le \left({{1}\over {k}}\sum_{i=1}^k \log
p_i\right)^k & \le  & \left({{\log(|u_{r^{k+4}}|)}\over {k}}\right)^k\\
& < & \left(\frac{r^{k+4}\log|\alpha_1|+\log 2}{k}\right)^k\\
& < & \left(\frac{2r^{k+4}\log|\alpha_1|}{k}\right)^k,
\end{eqnarray*}
for $k\ge 2$. In the above, we have also used the fact that 
$|u_n|<2|\alpha_1|^n$ holds for all $n\ge
1$ with the choice $n:=r^{k+4}$. Let $c_3:=2\log |\alpha_1|$. 
The above lower bound is
\begin{equation*}
\begin{split}
\#\cN_u(x)&\ge \left({{\log x}\over {r^{k+4} c_3}}+O\left({{k}\over
{r^k}}\right)\right)^k=\left({{\log x}\over {r^{k+4}
c_3}}\right)^k\left(1+O\left({{k^2}\over {\log x}}\right)\right)\\
& \gg \left({{\log x}\over {r^{k+4}
c_3}}\right)^k
\end{split}
\end{equation*}
provided that 
\begin{equation}
\label{eq:small k}
k=o({\sqrt{\log x}}), 
\end{equation}
as $x\to\infty$, which is now assumed. 
So, it suffices to look at
$$
\left({{\log x}\over {r^{k+4}c_3}}\right)^k= \exp\left(k\log(\log
x/c_3)-k(k+4)\log r\right).
$$
Let $A:=\log(\log x/c_3)$. In order to maximize the function
$f(x):=xA-x(x+4)\log r$, we take its derivative and set it equal to zero
to get $A-2x\log r-4\log r=0$, therefore $x=(A-4\log r)/(2\log r)=A/(2\log r)-2$. Thus, taking
$k:=\lfloor A/(2\log r)-2\rfloor$ (so that~\eqref{eq:small k}
is satisfied), we get that
$f(k)=f(x)+O(f'(x))=A^2/(4\log r)+O(A)$, hence
\begin{equation*}
\begin{split}
\#\cN_u(x)&\ge \exp\left({{(\log(\log x/c_3))^2}\over {4\log
r}}+O(\log\log x)\right)\\
&=\exp\left({{(\log\log x)^2}\over
{4\log r}}+O(\log\log x)\right),
\end{split}
\end{equation*}
which implies the desired conclusion with any constant $c_1<1/(4\log r)$.

\section{Remarks}

We end with a result showing that it is quite possible for $\#\cN_u(x)$ to be large under quite mild conditions. Observe that the sequence $u_n=2^n-2$ has the property that $u_1=0$. Here is a more general version of this fact.

\begin{prop}
\label{prop:last}
Let $k\ge 2$ and $\{u_n\}_{n\ge 0}$ be a linearly recurrent sequence of order $k$ satisfying recurrence 
\eqref{eq:LRS}. Assume that there exists a positive integer $n_0$ coprime to $a_k$ such that $u_{n_0}=0$. 
Then
$$
\#\cN_u(x)\gg x/\log x,
$$
where the implied constant  depends on the sequence $\{u_n\}_{n\ge 0}$.
\end{prop} 

\begin{proof}
Since $n_0$ is coprime to $a_k$, it follows that $\{u_n\}_{n\ge 0}$ is purely periodic modulo $n_0$. 
Let $t_{n_0}$ be this period. 
Now, let 
$\cR_u$  be the set of primes $p\equiv 1\pmod {t_{n_0}}$ such that $f_u(X)$ splits into linear factors modulo $p$. Alternatively, $\cR_u$ is the set of primes $p$ such that the polynomial $f_u(X)(X^{t_{n_0}}-1)$ splits into linear factors modulo $p$. The set of such primes has a positive density by the Chebotarev density theorem. We claim that 
\begin{equation}
\label{eq:inclusion}
\cS_u \subseteq \cN_u, 
\end{equation}
where
$$
\cS_u :=\{pn_0:p\in \cR_u~{\text{\rm and}}~p>n_0|\Delta_u|\}.
$$
The above inclusion implies the desired bound since then 
$$\#\cN_u(x)\ge \#\cR_u(x/n_0)+O(1)\gg x/\log x.$$ 

So, let us suppose that $p>n_0|\Delta_u|$ is in $\cR_u$. Then $p\equiv 1\pmod {t_{n_0}}$, therefore $p=1+\lambda t_{n_0}$ for some positive integer $\lambda$. Thus, $pn_0=n_0+\lambda n_0t_{n_0}$ and since $\{u_n\}_{n\ge 1}$ is purely periodic with period $t_{n_0}$ modulo $n_0$, we get that  
\begin{equation}
\label{eq:n0}
u_{pn_0}=u_{n_0+\lambda n_0 t_{n_0}}\equiv u_{n_0} \equiv 0\pmod {n_0}.
\end{equation}
Next, observe that since the polynomial $f_u(X)$ factors in linear factors modulo $p$, we get that $\alpha_i^{p}\equiv \alpha_i\pmod p$ for all $i=1,\ldots,k$. In particular, $\alpha_i^{pn_0}\equiv \alpha_i^{n_0}\pmod p$ for all $i=1,\ldots,k$. Since the denominators of the 
coefficient $A_i$, $i=1,\ldots,k$, in~\eqref{eq:un1} are divisors 
of $\Delta_u$ and $p>|\Delta_u|$, it follows that such denominators are invertible modulo $p$, therefore 
$A_i\alpha_i^{pn_0}\equiv A_i\alpha_i^{n_0}\pmod p$ for all $i=1,\ldots,k$. Summing up these congruences for $i=1,\ldots, k$, we get
\begin{equation}
\label{eq:p}
u_{pn_0}=\sum_{i=1}^k A_i \alpha_i^{pn_0}\equiv \sum_{i=1}^k A_i \alpha_i^{n_0} \equiv u_{n_0} \equiv 0\pmod p.
\end{equation}
{From} the congruences~\eqref{eq:n0} and~\eqref{eq:p}, we get that both $p$ and $n_0$ divide $u_{pn_0}$, and since $p$ is coprime to $n_0$, 
we get that $pn_0\mid u_{pn_0}$. This completes the proof of the 
inclusion~\eqref{eq:inclusion} and of the proposition.
\end{proof}

The condition that $n_0$ is coprime to $a_k$ is not always necessary. 
The conclusion of Propositon~\ref{prop:last} may hold without this condition like in the example of the sequence of general term
$$
u_n=10^{n}-7^n-2\cdot 5^n-1\qquad {\text{\rm for~all}}\quad n\ge 0,
$$
for which we can take $n_0=2$. Observe that $k=4$, 
$$
f_u(X)=(X-10)(X-7)(X-5)(X-1),
$$ 
and $n_0$ is not coprime to $a_4=-350$, yet one can check that the divisibility relation $2p\mid u_{2p}$ holds for all primes $p\ge 11$. We do not give further details. 

Let $\cM_u(x)$ be the set of integers $n\le x$ with $n\mid u_n$ and
$n$ is {\em not} of the form $pn_0$, where $p$ is prime and $u_{n_0}=0$.
It may be that in the situation of Theorem~\ref{thm:Div}, we can get a
smaller upper bound for $\#\cM_u(x)$ than for $\#\cN_u(x)$.
We can show this in a special case.

\begin{prop}
Assume that $\{u_n\}_{n\ge 0}$ is a linearly recurrent sequence of order 
$k$ whose characteristic polynomial splits into distinct linear factors in $\Z[X]$.
There is a positive constant $c_{4}(k)$ depending on $k$ such that for all sufficiently
large $x$ (depending on the sequence $u$), we have $\#\cM_u(x)\le x/L(x)^{c_{4}(k)}$.
\end{prop}
\begin{proof}
Let $y=L(x)$.  
We partition $\cM_u(x)$ into the following subsets:
\begin{eqnarray*}
\cM_1(x)&:=&\{n\in\cM_u(x):P(n)\le y\};\\
\cM_2(x)&:=&\{n\in\cM_u(x):\hbox{there is a prime }p\mid n,~p>y,~pT_u(p)\le kx\};\\
\cM_3(x)&:=&\cM_u(x)\setminus(\cM_1(x)\cup\cM_2(x)).
\end{eqnarray*}
As in the proof of Theorem~\ref{thm:2}, we see that 
Lemma~\ref{lem:Smooth} implies that
$\#\cM_1(x)\le x/L(x)^{1/2+o(1)}$
as $x\to\infty$.  

As in the proof of Theorem~\ref{thm:Div}, 
$$
\#\cM_2(x)\ll\sum_{\substack{y<p\le x\\ pT_u(p)\le kx}}\left(\frac{x}{pT_u(p)}+1
\right)
\ll \sum_{y<p\le x}\frac{x}{pT_u(p)}.
$$
We break this last sum according as $p\in\cP_{u,1/(k+1)}$ and $p\not\in\cP_{u,1/(k+1)}$, respectively.
Lemma~\ref{lem:3} shows that $\#\cP_{u,1/(k+1)}(t)\ll
t^{k/(k+1)}/\log t$.  Thus,
$$
\sum_{\substack{y<p\le x\\ p\in\cP_{u,1/(k+1)}}}\frac{x}{pT_u(p)}
\le
\sum_{\substack{y<p\le x\\ p\in\cP_{u,1/(k+1)}}}\frac{x}{p}
\ll \frac{x}{y^{1/(k+1)}}
$$
and
$$
\sum_{\substack{y<p\le x\\ p\not\in\cP_{u,1/(k+1)}}}\frac{x}{pT_u(p)}
\le \sum_{y<p\le x}\frac{x}{py^{1/(k+1)}}
\ll \frac{x\log_2x}{y^{1/(k+1)}}.
$$
Hence,
$$
\#\cM_2(x)\ll\frac{x}{L(x)^{1/(k+1)+o(1)}}\qquad {\text{\rm as}}\qquad x\to\infty.
$$

Suppose now that $n\in\cM_3(x)$.  Let $p\mid n$ with $pT_u(p)>kx$.
Using as before the notation $t_p$ for the period of $u$ modulo $p$, as well as the fact that
$T_u(p)\le kt_p$ and $t_p\mid p-1$ (since $f_u$ splits in linear factors over $\Q[X]$), we have
$$
kx<pT_u(p)\le kpt_p\le kp^2,
$$
so that $p>\sqrt{x}$.  Thus, $n$ can have at most one prime factor $p$
with $pT_u(p)>kx$.  So, if $n\in\cM_3(x)$, we may assume that $n=mp$
where $p>\sqrt{x} > m$, and $P(m)\le y$.  Further,
we may assume that $u_m\ne0$.  Since $p\mid u_{pm}$ and $t_p\mid p-1$,
we have $p\mid u_m$.  Now the number of prime factors of $u_m$ is $O(m)$.
Since the number of $n\in\cM_3(x)$ with such a prime $p\mid n$ is 
$O(x/(pT_u(p))+1)=O(1)$, we have
$$
\#\cM_3(x)\ll\sum_{\substack{m<\sqrt{x}\\P(m)\le y}}m
\le\sqrt{x}\Psi(\sqrt{x},y)=\frac{x}{L(x)^{1/4+o(1)}}\qquad {\text{\rm as}}\qquad x\to\infty,
$$ 
using Lemma~\ref{lem:Smooth}.  

We conclude that the result holds with $c_{4}(k):=\min\{1/5,1/(k+2)\}$, say.
\end{proof}

Finally, we note that for a given non-constant
polynomial $g(X)\in \Z[X]$ one can consider the more general set
$$
\cN_{u,g} : = \{n\ge 1\ :  \ g(n)\mid u_n\}.
$$
We fix some real $y < x^{1/2}$ and note that by the Brun sieve (see~\cite[Theorem~2.3]{HR}),
there are at most 
\begin{equation}
\label{eq:bound N1}
N_1  \ll x \left(\frac{\log y}{\log x}\right)
\end{equation}
values of $n\le x$ such that $g(n)$ does not have a 
prime divisor in the interval $[y,x^{1/2}]$. We also note that for a prime $p$ 
not dividing the content of $g$,
the divisibility $p \mid g(n)$ puts $n$ in at most 
$\deg g$ arithmetic progressions. 
Thus, using Lemma~\ref{lem:T} as it was used in the proof of Theorem~\ref{thm:Div},
the number of other $n\le x$ with $g(n)\mid u_n$
can be estimated as 
$$
N_2 \le \sum_{p \in [y,x^{1/2}]} \sum_{\substack{n \le x\\ p \mid g(n)\\ p \mid u_n}}1 
\ll \sum_{p \in [y,x^{1/2}]}  \(\frac{x}{pT_u(p)} + 1\) \ll 
x \sum_{p \in [y,x^{1/2}]} \frac{1}{pT_u(p)}  + O(x^{1/2}).
$$
Using Lemma~\ref{lem:3}  for any $\gamma \in (0,1)$ and the trivial estimate $T_u(p) \gg \log p$, we derive 
$$
\sum_{p \in [z,2z]} \frac{1}{pT_u(p)} \le 
\frac{1}{z}\sum_{p \in [z,2z]} \frac{1}{T_u(p)} 
\ll \frac{1}{z}\(\frac{z^{k\gamma}}{(\log z)^2}  + \frac{z^{1-\gamma}}{\log z} \).
$$ 
Taking $\gamma$ to satisfy
$$
z^\gamma = (z \log z)^{1/(k+1)},
$$
we obtain 
$$
\frac{1}{z}\sum_{p \in [z,2z]} \frac{1}{pT_u(p)} \ll  z^{-1/(k+1)}  (\log z)^{-(k+2)/(k+1)}.
$$
Summing over dyadic intervals, we now have 
$$
\sum_{p \in [y,x^{1/2}]} \frac{1}{pT_u(p)} \ll y^{-1/(k+1)}  (\log y)^{-(k+2)/(k+1)}.
$$
Therefore, 
\begin{equation}
\label{eq:bound N2}
N_2  \ll x y^{-1/(k+1)}  (\log y)^{-(k+2)/(k+1)} + x^{1/2}.
\end{equation}
Taking, for example, $y := (\log x)^{k+1}$, we obtain from~\eqref{eq:bound N1} 
and~\eqref{eq:bound N2}
 the estimate
\begin{equation}
\label{eq:bound Nug}
\# \cN_{u,g}(x) \le N_1+N_2\ll x\left( \frac{\log \log x}{\log x}\right).
\end{equation}

It is certainly an interesting question to bring the bound~\eqref{eq:bound Nug}
to the same shape as that of Theorem~\ref{thm:Div}. However, the method of proof 
of Theorem~\ref{thm:Div} does not apply due to the possible existence of large prime divisors
of $g(n)$. 

\section*{Acknowledgements} 

The authors are grateful to Chris Smyth who attracted their attention 
to the problems considered in this paper and to Larry Somer for 
many valuable comments. 

During the preparation of this paper, F.~L.\ was   supported in
part by grants SEP-CONACyT~79685 and PAPIIT~100508, C.~P. was supported in part by
NSF grant DMS-1001180 and I.~S.\ was supported in part by
ARC grant DP1092835.

\end{document}